\documentclass[draft]{amsart}
\usepackage{amsmath}
\usepackage{amssymb}

\newtheorem{thm}{Theorem}
\newtheorem{lem}[thm]{Lemma}
\newtheorem{cor}[thm]{Corollary}

\theoremstyle{remark}

\newcommand{\cv}{\mathbb{C}}

\newcommand{\rv}{\mathbb{R}}

\newcommand{\aut}{\textup{Aut}}
\def\Re{{\sf Re}\,}

\numberwithin{thm}{section}
\numberwithin{equation}{section}

\begin{document}

\title{Non-tangential Burns-Krantz rigidity}

\author{Feng Rong}

\address{Department of Mathematics, School of Mathematical Sciences, Shanghai Jiao Tong University, 800 Dong Chuan Road, Shanghai, 200240, P.R. China}
\email{frong@sjtu.edu.cn}



\thanks{The author is partially supported by the National Natural Science Foundation of China (Grant No. 12271350).}

\begin{abstract}
We prove some non-tangential Burns-Krantz type boundary rigidity theorems.
\end{abstract}

\maketitle

\section{Introduction}

In \cite{BK}, Burns and Krantz established the following fundamental rigidity result.

\begin{thm}\cite{BK}
Let $\Omega$ be a smoothly bounded strongly pseudoconvex domain in $\cv^n$, $n\ge 1$, $p\in \partial \Omega$, and $f$ be a holomorphic self-map of $\Omega$. If $f(z)=z+o(|z-p|^3)$ as $z\rightarrow p$, then $f(z)\equiv z$.
\end{thm}

Furthermore, they gave an example showing that the exponent 3 is optimal. Since then, many Burns-Krantz type rididity theorems have been obtained in various settings, see e.g. \cite{BZZ, BKR, EJLS, ELSS, ES, FR, FR1, GM, H:PJM93, H:IJM94, H:CJM95, KMac, LR, LRS, LS, Sh, TV, Z:rigidity} and the references therein.

In this paper, we study two questions regarding the Burns-Krantz rigidity.\\
{\bf Q1}: Is the rigidity valid if $f(z)=z+o(|z-p|^3)$ as $z\rightarrow p$ non-tangentially?\\
{\bf Q2}: What happens if $f$ has an interior fixed point?

Question 1 is quite natural, since a key ingredient for proving the Burns-Krantz rigidity is Hopf's lemma, which only concerns the radial limit of $f'$. Regarding this question, we mention the following two results.

\begin{thm}\cite[Proposition 1]{KMac}\label{T:KMac}
Let $f$ be a holomorphic self-map of $\Delta$ such that $f(\zeta)=\zeta+o(\zeta-1)$ as $\zeta\rightarrow 1$ non-tangentially. If
$$\liminf_{r\rightarrow 1^-} \frac{\Re(f(r)-r)}{(1-r)^3}=0,$$
then $f(\zeta)\equiv \zeta$.
\end{thm}

\begin{thm}\cite[Corollary 1.2]{BZZ}\label{T:BZZ}
Let $M\subset \cv^n$, $n\ge 2$, be a strongly pseudoconvex hypersurface at a point $p\in M$ and $f$ be a germ at $p$ of a holomorphic self-map of the pseudoconvex side such that $f(z)=z+o(|z-p|^3)$ as $z\rightarrow p$ non-tangentially. Then $f(z)\equiv z$.
\end{thm}

Question 2 is also natural, since the boundary Schwarz lemma has two versions with or without interior fixed point, and the Burns-Krantz theorem is the rigidity part of the boundary Schwarz lemma. Regarding this question, we mention the following two results.

\begin{thm}\cite[Corollary 1.5]{H:CJM95}\label{T:HD}
Let $D$ be a smoothly bounded domain in $\cv$, $p\in \partial D$, and $f$ be a holomorphic self-map of $D$. If $f(z_0)=z_0$ for some $z_0\in D$ and $f(z)=z+o(|z-p|)$ as $z\rightarrow p$, then $f(z)\equiv z$.
\end{thm}

\begin{thm}\cite[Corollary 2.7]{H:CJM95}\label{T:HO}
Let $\Omega$ be a smoothly bounded strongly convex domain in $\cv^n$, $n\ge 2$, $p\in \partial\Omega$, and $f$ be a holomorphic self-map of $\Omega$. If $f(z_0)=z_0$ for some $z_0\in \Omega$ and $f(z)=z+o(|z-p|^2)$ as $z\rightarrow p$, then $f(z)\equiv z$.
\end{thm}

The exponent 2 in Theorem \ref{T:HO} is optimal as an example in \cite{H:CJM95} shows. Moreover, it was conjectured in \cite{H:CJM95} that a similar result should hold for bounded strongly pseudoconvex domains. This was recently proven by the author in \cite{R} (see also \cite{H:IJM94, FR} for partial results in $\cv^2$).

In section \ref{S:Delta}, we prove some non-tangential boundary rigidity results for domains in $\cv$, namely, Theorem \ref{T:Delta}, Corollary \ref{C:Delta}, Theorem \ref{T:D} and Corollary \ref{C:D}. In section \ref{S:O}, we establish non-tangential boundary rigidity results for strongly linearly convex domains in $\cv^n$, $n\ge 2$, namely, Theorem \ref{T:O} and Corollary \ref{C:O}. A few more non-tangential boundary rigidity results were given for strongly pseudoconvex domains and ``fibered domains" in section \ref{S:other}.

\section{Domains in $\cv$}\label{S:Delta}

Let us first recall some definitions and results from the iteration theory on the unit disk $\Delta$. The main reference here is \cite{A:book}.

The \textit{horocycle} of center $\sigma\in \partial \Delta$ and radius $R>0$ is the set
$$E(\sigma,R)=\left\{\zeta\in \Delta;\ \frac{|\sigma-\zeta|^2}{1-|\zeta|^2}<R\right\}.$$

\begin{lem}[Julia's lemma]
Let $f$ be a holomoprhic self-map of $\Delta$. Suppose that $\sigma\in \partial \Delta$ is a boundary fixed point of $f$, i.e. $f(\zeta)\rightarrow \sigma$ as $\zeta\rightarrow \sigma$. Set
$$\alpha=\liminf_{D\ni \zeta\rightarrow \sigma} \frac{1-|f(\zeta)|}{1-|\zeta|}.$$
Then, $f(E(\sigma,R))\subset E(\sigma,\alpha R)$ for any $R>0$, i.e.,
$$\frac{|\sigma-f(\zeta)|^2}{1-|f(\zeta)|^2}\le \alpha\frac{|\sigma-\zeta|^2}{1-|\zeta|^2},\ \ \ \zeta\in \Delta.$$
\end{lem}

\begin{lem}[Wolff's lemma]
Let $f$ be a holomorphic self-map of $\Delta$. If $f$ has no interior fixed point, then there exists an unique point $\sigma\in \partial \Delta$, called the \textit{Wolff point} of $f$, such that $f(E(\sigma,R))\subset E(\sigma,R)$ for any $R>0$, i.e.,
$$\frac{|\sigma-f(\zeta)|^2}{1-|f(\zeta)|^2}\le \frac{|\sigma-\zeta|^2}{1-|\zeta|^2},\ \ \ \zeta\in \Delta.$$
\end{lem}

\begin{thm}[Wolff-Denjoy theorem]
Let $f$ be a holomorphic self-map of $\Delta$. Then, $f$ has no interior fixed point if and only if $f^k\rightarrow \sigma$ normally on $\Delta$, where $\sigma\in \partial \Delta$ is the Wolff point of $f$.
\end{thm}

Combining ideas from \cite[Lemma 2.1]{H:CJM95} and Theorem \ref{T:KMac}, we have the following non-tangential rigidity result.

\begin{thm}\label{T:Delta}
Let $f$ be a holomorphic self-map of $\Delta$ such that $f(\zeta)=\zeta+o(\zeta-1)$ as $\zeta\rightarrow 1$ non-tangentially. Then, either $f(\zeta)\equiv \zeta$, or $1$ is the Wolff point of $f$ and
$$f^k\rightarrow 1,\ \ \ \liminf_{r\rightarrow 1^-} \frac{\Re(f(r)-r)}{(1-r)^3}>0.$$
\end{thm}
\begin{proof}
Set $\varphi(r):=f(r)-r$. Then $\varphi(r)=o(1-r)$ as $r\rightarrow 1^-$. Assume that $f(\zeta)\not\equiv \zeta$. First, we show that $f$ has no interior fixed points. Otherwise, without loss of generality, suppose that $f(0)=0$. Then, by the Schwarz lemma, $|f(\zeta)|<|\zeta|$ for $\zeta\neq 0$. Thus,
$$v(\zeta):=\Re \left(\frac{f(\zeta)-\zeta}{f(\zeta)+\zeta}\right)=\frac{|f(\zeta)|^2-|\zeta|^2}{|f(\zeta)+\zeta|^2}<0,\ \ \ \zeta\in \Delta\backslash\{0\}=:\Delta^\ast.$$
Note that $v$ is a negative harmonic function on $\Delta^\ast$, satisfying $v(r)=o(1-r)$ as $r\rightarrow 1^-$. In particular, $v$ takes the maximum $0$ at $\zeta=1$ and is $o(1-r)$ there radially. Thus $v\equiv 0$, by Hopf's lemma. But this would imply that $f(\zeta)\equiv \zeta$, a contradiction.

Now, by the Wolff-Denjoy theorem, $f^k\rightarrow \sigma(f)$, where $\sigma(f)$ is the Wolff point of $f$. Since $f(\zeta)=\zeta+o(\zeta-1)$ as $\zeta\rightarrow 1$ non-tangentially, Julia's lemma shows that $f(E(1,R))\subset E(1,R)$ for any $R>0$, i.e.,
$$\frac{|1-f(\zeta)|^2}{1-|f(\zeta)|^2}\le \frac{|1-\zeta|^2}{1-|\zeta|^2},\ \ \ \zeta\in \Delta.$$
Thus, $\sigma(f)=1$ and $f^k\rightarrow 1$. Set
$$u(\zeta):=\Re \left(\frac{1+\zeta}{1-\zeta}-\frac{1+f(\zeta)}{1-f(\zeta)}\right)\le 0,\ \ \ \zeta\in \Delta,$$
and
$$\psi(r):=\frac{\varphi(r)}{1-r}=s(r)+it(r),\ \ \ \tilde{r}:=\frac{1+r}{1-r}.$$
Then, $u(\zeta)$ is a negative harmonic function on $\Delta$ and
$$u(r)=\Re \left(\frac{1+r}{1-r}-\frac{1+f(r)}{1-f(r)}\right)=\tilde{r}-\Re \frac{\tilde{r}+\psi(r)}{1-\psi(r)}=\frac{(1+\tilde{r})(s^2+t^2-s)}{(1-s)^2+t^2}.$$
If $\liminf_{r\rightarrow 1^-} \frac{\Re(f(r)-r)}{(1-r)^3}=0$, then
$$\liminf_{r\rightarrow 1^-}\frac{-u(r)}{1-r}=\liminf_{r\rightarrow 1^-} \frac{2}{(1-r)^2}\frac{s-s^2-t^2}{(1-s)^2+t^2}\le \liminf_{r\rightarrow 1^-} \frac{s}{(1-r)^2}\frac{2(1-s)}{(1-s)^2+t^2}=0.$$
Thus $u\equiv 0$, by Hopf's lemma. But this would imply that $f(\zeta)\equiv \zeta$, a contradiction.
\end{proof}

As an immediate corollary, we have

\begin{cor}\label{C:Delta}
Let $f$ be a holomorphic self-map of $\Delta$ such that $f(\zeta)=\zeta+o(\zeta-1)$ as $\zeta\rightarrow 1$ non-tangentially. If $f$ has an interior fixed point then $f(\zeta)\equiv \zeta$.
\end{cor}

To treat more general domains in $\cv$, we need the following lemma. It is a generalized form of \cite[Lemma 8]{FR1}, with essentially the same proof.

\begin{lem}\label{L:order}
Let $D$ be a domain in $\cv$ with $C^2$-smooth boundary near $p\in \partial D$. Let $\phi:\Delta\rightarrow D$ be either a Riemann mapping or a universal cover with $\lim_{k\rightarrow \infty} \phi(\zeta_k)=p$ for some $\zeta_k\rightarrow 1$. Then $\phi$ extends to be bi-Lipschitz near $1\in \overline{\Delta}$.
\end{lem}

Using Lemma \ref{L:order}, we see that any holomorphic self-map $f$ of $D$ with $f(z)=z+o(|z-p|)$ as $z\rightarrow p$ non-tangentially can be lifted to a holomorphic self-map $\tilde{f}$ of $\Delta$ with $\tilde{f}(\zeta)=\zeta+o(\zeta-1)$ as $\zeta\rightarrow 1$ non-tangentially. Thus, we can state Theorem \ref{T:Delta} and Corollary \ref{C:Delta} for more general domains in $\cv$ as follows.

\begin{thm}\label{T:D}
Let $D$ be a domain in $\cv$ with $C^2$ boundary near $p\in \partial D$, and $f$ be a holomorphic self-map of $D$. If $f(z)=z+o(|z-p|)$ as $z\rightarrow p$ non-tangentially, then either $f(z)\equiv z$, or
$$f^k\rightarrow p,\ \ \ \limsup_{z\rightarrow p\ \textup{non-tangentially}} \frac{|f(z)-z|}{|z-p|^3}>0.$$
\end{thm}

\begin{cor}\label{C:D}
Let $D$ be a domain in $\cv$ with $C^2$ boundary near $p\in \partial D$, and $f$ be a holomorphic self-map of $D$. If $f(z_0)=z_0$ for some $z_0\in D$ and $f(z)=z+o(|z-p|)$ as $z\rightarrow p$ non-tangentially, then $f(z)\equiv z$.
\end{cor}

Note that Corollary \ref{C:D} is a local and non-tangential version of Theorem \ref{T:HD}.

\section{Strongly linearly convex domains}\label{S:O}

A bounded domain $\Omega$ in $\cv^n$, $n\ge 2$, is \textit{strongly linearly convex} if it has a $C^2$-smooth boundary and admits a $C^2$-defining function $r:\cv^n\rightarrow \rv$, whose real Hessian is positive definite on the complex tangent space of $\partial \Omega$, i.e.,
$$\sum_{j,k=1}^n \frac{\partial^2 r}{\partial z_j\partial \bar{z}_k}(p)v_j\bar{v}_k>\left|\sum_{j,k=1}^n \frac{\partial^2 r}{\partial z_j\partial z_k}(p)v_jv_k\right|,$$
for all $p\in \partial \Omega$ and non-zero $v=(v_1,\cdots,v_n)\in T_p^{1,0}\partial \Omega$ (see e.g. \cite{Ai, APS, M, Zn}). The notion of strong linear convexity is weaker than the usual strong convexity but stronger than strong pseudoconvexity.

For any bounded domain $\Omega$ and $z_1,z_2\in \Omega$, denote by $k_\Omega(z_1,z_2)$ the \textit{Kobayashi distance} on $\Omega$. A holomorphic map $\varphi:\Delta\rightarrow \Omega$ is called a \textit{complex geodesic} of $\Omega$ (\cite{V}), if it is an isometry between $k_\Delta$ and $k_\Omega$, i.e. $k_\Omega(\varphi(\zeta_1),\varphi(\zeta_2))= k_\Delta(\zeta_1,\zeta_2)$ for all $\zeta_1,\zeta_2\in \Delta$.

The existence of complex geodesics with prescribed interior data on strongly convex domains was established by Lempert (\cite{L:BSMF}). Subsequent work by various authors show that similar results also hold with prescribed boundary data on strongly linearly convex domains (see e.g. \cite{A:book, CHL, H:SNS94, HW, L2}). We collect some known facts about complex geodesics on strongly linearly convex domains in the following theorem.

\begin{thm}\label{T:geodesic}
Let $\Omega$ be a bounded strongly linearly convex domain with $C^k$ boundary, $k\ge 3$. Then,\\
\textup{1)} for every $q\in \Omega$ and $p\in \partial \Omega$, there exists a unique complex geodesic $\varphi$ with $\varphi(0)=q$ and $\varphi(1)=p$, and $\varphi\in C^{k-2}(\overline{\Delta})$;\\
\textup{2)} for every complex geodesic $\varphi:\Delta\rightarrow \Omega$, there exists $\psi:\Omega\rightarrow \Delta$ such that $\psi\circ \varphi=\textup{id}_\Delta$ and $\psi\in C^{k-2}(\overline{\Omega})$;\\
\textup{3)} every complex geodesic on $\Omega$ is transversal to $\partial \Omega$;\\
\textup{4)} if $\zeta\rightarrow 1$ non-tangentially on $\Delta$, then $\varphi(\zeta)\rightarrow p$ non-tangentially on $\Omega$.
\end{thm}

The non-tangential version of the Burns-Krantz rigidity theorem for strongly linearly convex domains below obviously follows from Theorem \ref{T:BZZ}. It can also be proven following the original arguments of Burns and Krantz in \cite{BK}, using Theorem \ref{T:geodesic} and Theorem \ref{T:Delta} instead of the Burns-Krants rigidity on $\Delta$.

\begin{thm}\label{T:O1}
Let $\Omega$ be a bounded strongly linearly convex domain with $C^3$ boundary, $p\in \partial \Omega$, and $f$ be a holomorphic self-map of $\Omega$. If $f(z)=z+o(|z-p|^3)$ as $z\rightarrow p$ non-tangentially, then $f(z)\equiv z$.
\end{thm}

The next result is the non-tangential version of \cite[Theorem 2.2]{H:CJM95}. The proof of \cite[Theorem 2.2]{H:CJM95} can be carried over almost verbatim, and we omit the details.

\begin{thm}\label{T:Huang}
Let $\Omega$ be a bounded domain. Let $\varphi:\Delta\rightarrow \Omega$ be a proper holomorphic map which is $\epsilon$-H\"{o}lder continuous up to the boundary with $\varphi(1)=p\in \partial\Omega$ and transversal to $\partial \Omega$. Suppose that there exists a defining function $\rho$ of $\Omega$ which is smooth over $\varphi(\overline{\Delta})$ with $d\rho\neq 0$ on $\varphi(\partial \Delta)$ and that there exists a positive number $\mu\le 1$ such that $-(-\rho\circ \varphi)^\mu$ is subharmonic. Then, if $f$ is a holomorphic self-map of $\Omega$ fixing $\varphi(\Delta)$, and if $f(z)=z+o(|z-p|^{2/\mu\epsilon})$ as $z\rightarrow p$ non-tangentially, then $f(z)\equiv z$.
\end{thm}

For any $z_0\in \Omega$, $p\in \partial \Omega$ and $R>0$, the \textit{small horosphere} $E_{z_0}(p,R)$ and the \textit{big horosphere} $F_{z_0}(p,R)$ of center $p$, pole $z_0$ and radius $R$ are defined by
$$\begin{aligned}
E_{z_0}(p,R)&=\left\{z\in \Omega:\ \limsup_{w\rightarrow p} [k_\Omega(z,w)-k_\Omega(z_0,w)]<\frac{1}{2}\log R\right\},\\
F_{z_0}(p,R)&=\left\{z\in \Omega:\ \liminf_{w\rightarrow p} [k_\Omega(z,w)-k_\Omega(z_0,w)]<\frac{1}{2}\log R\right\}.
\end{aligned}$$

We need the following two results.

\begin{lem}\cite[Proposition 1.2]{H:CJM95}\label{L:1.2}
Let $\Omega$ be a bounded domain with $C^1$ boundary, $p\in \partial \Omega$, and $f$ be a holomorphic self-map of $\Omega$. If $f(z)=z+o(|z-p|)$ as $z\rightarrow p$ non-tangentially, then $f(E_{z_0}(p,R))\subset F_{z_0}(p,R)$ for all $z_0\in \Omega$ and $R>0$.
\end{lem}

\begin{lem}\cite[Lemma 2.2]{BZ:jets}\label{L:jets}
Let $\Omega$ be a bounded domain with $C^1$ boundary, $p\in \partial \Omega$, and $f:\Omega\rightarrow \cv^m$ be holomorphic. If $df_p$ is the non-tangential differential of $f$ at $p$, then $df_z\rightarrow df_p$ as $z\rightarrow p$ non-tangentially.
\end{lem}

The following corollary is the non-tangential version of \cite[Corollary 1.3]{H:CJM95}.

\begin{cor}\label{C:1.3}
Let $\Omega$ be a bounded domain with $C^1$ boundary, $p\in \partial \Omega$, and $f$ be a holomorphic self-map of $\Omega$. Suppose that there exist a point $z_0\in \Omega$ and $R>0$ such that $E_{z_0}(p,R)\neq \emptyset$ and $\overline{F_{z_0}(p,R)}\cap \partial \Omega=\{p\}$. If $f(z)=z+o(|z-p|)$ as $z\rightarrow p$ non-tangentially and $\{f^k\}$ is compactly divergent, then $f^k\rightarrow p$.
\end{cor}
\begin{proof}
Let $g$ be a limit point of $\{f^k\}$. Then, $g$ is a holomorphic map from $\Omega$ to $\partial \Omega$ since $\{f^k\}$ is compactly divergent. Since $df_p=id$, by Lemma \ref{L:jets}, if $z\rightarrow p$ non-tangentially then $f(z)\rightarrow p$ non-tangentially. Thus, for each $k$, $f^k$ satisfies the the hypotheses in Lemma \ref{L:1.2}. Thus, $f^k(E_z(p,R))\subset F_z(p,R)$ for every $z\in \Omega$ and $R>0$. Hence, $g(E_{z_0}(p,R))\subset \overline{F_{z_0}(p,R)}\cap \partial \Omega=\{p\}$. This implies that $g\equiv p$. Therefore, $f^k\rightarrow p$.
\end{proof}

We need two more results.

\begin{thm}\cite[Theorem 1.7]{A:horo}\label{T:horo}
Let $\Omega$ be a bounded strongly pseudoconvex domain with $C^2$ boundary. Then for every $z_0\in \Omega$, $p\in \partial \Omega$ and $R>0$, $\overline{F_{z_0}(p,R)}\cap \partial \Omega=\{p\}$.
\end{thm}

\begin{thm}\cite[Theorem 1]{H:SNS94}
Let $\Omega$ be a bounded contractible strongly pseudoconvex domain with $C^3$ boundary, and $f$ be a holomorphic self-map of $\Omega$. Then, $f$ has no interior fixed points if and only if the iterates of $f$ converges uniformly on compacta to a boundary point, the Wolff point of $f$.
\end{thm}

Since a strongly linear convex domain is always contractible (\cite[Theorem 2.4.2]{APS}), we have the following

\begin{cor}\label{C:WD}
Let $\Omega$ be a bounded strongly linearly convex domain with $C^3$ boundary, and $f$ be a holomorphic self-map of $\Omega$. Then, $f$ has no interior fixed points if and only if the iterates of $f$ converges uniformly on compacta to a boundary point, the Wolff point of $f$.
\end{cor}

The following is the non-tangential version of \cite[Theorem 2.6]{H:CJM95}.

\begin{thm}\label{T:O}
Let $\Omega$ be a bounded strongly linearly convex domain with $C^3$ boundary, $p\in \partial \Omega$, and $f$ be a holomorphic self-map of $\Omega$. If $f(z)=z+o(|z-p|^2)$ as $z\rightarrow p$ non-tangentially, then either $f(z)\equiv z$, or $p$ is the Wolff point of $f$ and
$$f^k\rightarrow p,\ \ \ \limsup_{z\rightarrow p\ \textup{non-tangentially}} \frac{|f(z)-z|}{|z-p|^3}>0.$$
\end{thm}
\begin{proof}
Assume that $f(z)\not\equiv z$. First, we show that $f$ has no interior fixed points. Suppose that $f(q)=q$, $q\in \Omega$. Let $\varphi$ and $\psi$ be the maps given in Theorem \ref{T:geodesic}. Set $g=\psi\circ f\circ \varphi$. Then, as $\zeta\rightarrow 1$ non-tangentially, we have
$$g(\zeta)=\psi(\varphi(\zeta)+o(|\varphi(\zeta)-p|^2))=\psi(\varphi(\zeta))+o(|\varphi(\zeta)-\varphi(1)|^2)=\zeta+o(|\zeta-1|^2).$$
Noting that $g(0)=0$, we get $g(\zeta)\equiv \zeta$ by Corollary \ref{C:Delta}.

From the monotonicity of the Kobayashi distance, it follows that $f\circ \varphi$ is also a complex geodesic. By the uniqueness property, we have $f\circ \varphi\equiv \varphi$, i.e. $f$ fixes $\varphi(\Delta)$. From Theorem \ref{T:Huang} with $\epsilon,\mu=1$, it follows that $f(z)\equiv z$. This is a contradiction.

Since $f$ has no interior fixed points, it follows from Corollary \ref{C:1.3}, Theorem \ref{T:horo} and Corollary \ref{C:WD} that $f^k\rightarrow p$. Thus, since $f(z)\not\equiv z$,
$$\limsup_{z\rightarrow p\ \textup{non-tangentially}} \frac{|f(z)-z|}{|z-p|^3}>0,$$
by Theorem \ref{T:O1}.
\end{proof}

As an immediate corollary, we have the following non-tangential version of Theorem \ref{T:HO}.

\begin{cor}\label{C:O}
Let $\Omega$ be a bounded strongly linearly convex domain with $C^3$ boundary, $p\in \partial \Omega$, and $f$ be a holomorphic self-map of $\Omega$. If $f(z_0)=z_0$ for some $z_0\in \Omega$ and $f(z)=z+o(|z-p|^2)$ as $z\rightarrow p$ non-tangentially, then $f(z)\equiv z$.
\end{cor}

\section{Some other domains}\label{S:other}

Let $\Omega$ be a bounded domain in $\cv^n$, $n\ge 1$. Recall that an element $f\in \aut(\Omega)$ is called \textit{elliptic} if the closed subgroup of $\aut(\Omega)$ generated by $f$ is compact. The following is the non-tangential version of \cite[Theorem 1]{H:PJM93}. Since the original proof uses only the non-tangential limit, it also works here and we omit the details.

\begin{thm}
Let $\Omega$ be a bounded pseudoconvex domain satisfying condition $R$ and $p\in \partial \Omega$. If $f\in \aut(\Omega)$ is elliptic and satisfies $f(z)=z+o(|z-p|)$ as $z\rightarrow p$ non-tangentially, then $f(z)\equiv z$.
\end{thm}

Since we do not actually need condition $R$ in the sequel, we omit its precise definition. We only recall that strongly pseudoconvex domains satisfy condition $R$. And, by the Wong-Rosay theorem (see e.g. \cite{GKK}), every automorphism of a bounded strongly pseudoconvex domain not biholomorphic to the unit ball is elliptic. Thus, we have the following non-tangential version of a result due to Krantz.

\begin{cor}\label{C:Aut}
Let $\Omega$ be a bounded strongly pseudoconvex domain not biholomorphic to the unit ball and $p\in \partial \Omega$. If $f\in \aut(\Omega)$ satisfies $f(z)=z+o(|z-p|)$ as $z\rightarrow p$ non-tangentially, then $f(z)\equiv z$.
\end{cor}

In \cite{FR1}, Burns-Krantz type rigidity results were proven for some fibered domains. In the proof of the main results from \cite{FR1}, the key ingredients are the original Burns-Krantz rigidity on the unit disk, \cite[Theorem 2.2]{H:CJM95} and \cite[Corollary 1.5]{H:CJM95}. Since we have proven the non-tangential version of these results in the previous two sections, we can get the non-tangential version of the main results from \cite{FR1}.

First, consider bounded domains in $\cv^{n+1}$ of the form
$$\Omega=\bigcup\limits_{z\in \Delta} \Omega_z,$$
where $\Omega_z$'s are bounded domains in $\cv^n$ containing the origin. Let $\rho:\Delta\rightarrow \rv^+$ be a function such that $\Omega_z\subset \Delta_{\rho(z)}^n$ for any $z\in \Delta$ and $\rho\in C^0(\bar{\Delta})$. We say that $\Omega$ is a \textit{fibered domain with boundary size zero} if there exists a $\rho(z)$ with
$$\int_{-\pi}^\pi \log \rho(e^{i\theta}) d\theta =-\infty.$$

\begin{thm}
Let $\Omega$ be a fibered domain with boundary size zero and $f$ be a holomorphic self-map of $\Omega$. If $f(z,w)=(z,w)+o(\|(1-z,w)\|^3)$ as $(z,w)\rightarrow (1,0)$ non-tangentially, then $f(z,w)\equiv (z,w)$.
\end{thm}

\begin{thm}
Let $\Omega$ be a fibered domain with boundary size zero and $f$ be a holomorphic self-map of $\Omega$. If $f(z_0,0)=(z_0,0)$ for some $z_0\in \Delta$ and $f(z,w)=(z,w)+o(\|(1-z,w)\|^2)$ as $(z,w)\rightarrow (1,0)$ non-tangentially, then $f(z,w)\equiv (z,w)$.
\end{thm}

Next, consider domains in $\cv^{n+m}$ of the form
$$\Omega=\bigcup\limits_{z\in D} \Omega_z,$$
where $\Omega_z$'s are bounded complete Reinhardt domains in $\cv^n$ and $D$ is a domain in $\cv^m$ such that the non-tangential Burns-Krantz rigidity holds at some $p\in \partial D$. Let $\delta:D \rightarrow \rv^+$ be a function such that $\Delta_{\delta(z)}^n\subset \Omega_z$ for any $z\in D$ and $\delta(z)\rightarrow 0$ as $z\rightarrow p$. For any $(z,w)\in \Omega$, denote by $D_{z,w}$ the linear graph over $D$ through $(z,w)$ and $(p,0)$. We define a \textit{cone} with end $(p,0)$ (of size $\delta$) as
$$C_\delta:=\bigcup\{D_{z,w};\ z\in D,\ w\in \Delta_{\delta(z)}^n\}.$$
We say that $\Omega$ is a \textit{fibered domain satisfying the cone condition} if it contains a cone $C_\delta$.

\begin{thm}
Let $\Omega$ be a fibered domain satisfying the cone condition and $f$ be a holomorphic self-map of $\Omega$. If $f(z,w)=(z,w)+o(\|(z-p,w)\|^3)$ as $(z,w)\rightarrow (p,0)$ non-tangentially, then $f(z,w)\equiv (z,w)$.
\end{thm}

\begin{thm}
Let $\Omega$ be a fibered domain satisfying the cone condition and $f$ be a holomorphic self-map of $\Omega$. Suppose that $f(z,w)=(z,w)+o(\|(z-p,w)\|^2)$ as $(z,w)\rightarrow (p,0)$ non-tangentially and that the fixed point set $\Gamma$ of $f$ satisfies
$$\Gamma\cap D_{z,w}\neq \emptyset,\ \forall D_{z,w}\subset C_\delta.$$
Then, $f(z,w)\equiv (z,w)$.
\end{thm}

\end{document}